\documentclass[a4paper,12pt]{article}

\usepackage[cp1252]{inputenc}
\usepackage{amsmath}
\usepackage{amsfonts}
\usepackage{amssymb}
\usepackage{amsthm}
\usepackage{fontenc}
\usepackage[pdftex]{graphicx}
\usepackage[left=2.5cm,right=2.5cm,top=2cm,bottom=2cm]{geometry}
\usepackage{cite}
\usepackage{makeidx}
\usepackage{graphicx}
\usepackage{tikz}
\usepackage{pgf}

\newcommand{\R}{\mathbb{R}}

\newcommand{\diam}{\mathrm{diam}}
\newcommand{\dist}{\mathrm{dist}}

\newtheorem{definition}{Definition}
\newtheorem{theorem}{Theorem}

\newtheorem{remark}{Remark}

\newcommand{\ux}{\underline{x}}
\newcommand{\uy}{\underline{y}}

\newcommand{\E}{{\bf E}}

\begin{document}

\title{Reduction Procedure for obtaining solutions of the scalar additive Jump problem and Riemann Boundary Value Problem in vectorial Clifford analysis}

\author{Carlos Daniel Tamayo-Castro$^{1}$, Juan Bory-Reyes$^{2}$, Ricardo Abreu-Blaya$^{3}$}
\date { \small{$^{1}$Facultad de Ciencias. Universidad Nacional Aut\'onoma de M\'exico, Ciudad M\'exico, M\'exico.\\ Email: cdtamayoc@ciencias.unam.mx\\
		$^{2}$SEPI-ESIME, Instituto Polit\'ecnico Nacional, Ciudad de M\'exico, M\'exico\\ Email: juanboryreyes@yahoo.com \\
	    $^{3}$Facultad de Matem\'aticas, Universidad Aut\'onoma de Guerrero, Chilpancingo, M\'exico}\\ Email: rabreublaya@yahoo.es
}
\maketitle

\begin{abstract} 
In this paper, we study the existence of solutions to the scalar additive Jump problem and the Riemann boundary value problems in the context of vectorial Clifford analysis on domains with fractal boundaries. A reduction procedure is applied with great effectiveness to find the solution of the problems.
\vspace{0.3cm}

\small{

}  
\end{abstract}

\section{Introduction}
The Riemann boundary value problem (RBVP for short) for holomorphic functions in a bounded Jordan domain of the complex plane has been largely studied and applied to many branches of mathematics, physics, and engineering, see \cite{Gajov, Lu, Mu} for extensive treatments and discussions.

In the solution to the RBVP, the main mathematical tool is the Cauchy type integral. Even though it loses its meaning on fractal curves, the RBVP remains completely valid. A seminal work in the field was presented by Boris Kats in \cite{BKats83}, and in his subsequent publications \cite{Kats2009,Kats2014}. In particular, solvability conditions involving the upper Minkowski dimension of the domain's boundary and the H\"older exponent of the data function associated with the problem were considered. For a brief summary of this approach, we refer the reader to \cite{AbreuBlaya2011, BlayaReyesKats}.

In \cite{BDKats16, 1DKats16, 2DKats16}, the Marcinkiewicz exponent in the context of complex analysis was introduced. Using this new metric characteristic of the boundary, the solvability of the RBVPs were considerably sharped.

Clifford analysis generalizes, in a natural way, complex analysis to higher dimensions. One of its main approaches studies functions defined at $\mathbb{R}^{n}$ with values in the Clifford algebras $\mathcal{C}\ell(n)$ vanishing the Dirac operator. We follow \cite{TAB21}, assuming the so-called vectorial Clifford analysis approach. On the other hand, we will refer to the paravectorial Clifford analysis approach that studies functions defined at $\mathbb{R}^{n + 1}$ and valued in $\mathcal{C}\ell(n)$, mainly those that nullify the Clifford Cauchy-Riemann operator. The reader is referred to  \cite{BDS, GHS08, Gilbert1991} for a standard account of the theory.

As was pointed out in \cite[pp 22--24]{Ryan96}, there exist considerable obstacles for a complete treatment of the Riemann boundary value problem for monogenic functions. These are a direct consequence of the lack of commutativity in the Clifford product. This fact explains why an explicit solution to the Riemann boundary value problems has been found only for the particular case of the Jump problem and some slight modifications, see \cite{AB2001,Brs2001}. It is worth noting the recent work \cite{BoryReyes2016}, where a class of compound boundary value problems for the homogeneous Dirac equation in two and three dimensions was studied when one of the two boundary conditions is loaded. It is shown how the absence of commutativity in Clifford algebra paradoxically relaxes the conditions that guarantee the solvability of considered problems.

In the Clifford analysis setting, the solvability of the Riemann boundary value problem over fractal domains is mainly analyzed on the case of the scalar additive Jump problem. The original impetus of the current investigation on the plane involving the Marcinkiewicz exponent has been generalized recently to higher dimensions in \cite{TamayoCastro2022}.

The reduction procedure for the Riemann boundary value problem when the boundary is a smooth surface and relative to the elements $e_1$ and $e_n$ was developed in \cite{TAB21}. Particularly, to guarantee the existence and uniqueness of a solution to the RBVP in lower dimensional vectorial Clifford analysis were analyzed. This study showed that the problem solubility table may be accomplished in a variety of ways: no solution or exactly one solution or infinitely many linearly independent solutions.

Our purpose in this paper is to present a more general form of the reduction procedure for the Riemann boundary value problem, and then to apply this approach to the fractal setting. 

The manuscript is organized as follows. Section \ref{SecPreliminars} presents the necessary background about Clifford analysis, fractal geometric characteristics of the boundaries, and states a Whitney type extension theorem. In Section \ref{SecReductionProcedure}, a more general form of the reduction procedure for the Riemann boundary value problem is presented. In Section \ref{SecApplicationsFractal}, the main results of this work are developed. First, unique solvability conditions are proved for the scalar additive Jump problem in the vectorial Clifford analysis using the reduction procedure and similar results in the paravectorial approach. The effectiveness of the method has been proved in the sense that we get more expressions for the solutions than those we obtain when the method developed in \cite{TamayoCastro2022} in the paravectorial setting is reproduced. Second, the procedure is applied to solve the Riemann boundary value problem for a class of variable coefficients in lower-dimensional vectorial Clifford analysis.

\section{Preliminaries and Notations} \label{SecPreliminars}
This section contains essential background required to present the results in the subsequent sections. 
\subsection{Clifford Algebras and Monogenic Functions}
Rudiments of Clifford algebras and Clifford analysis are reviewed in this subsection. A discussion of the subjects can be found at \cite{BDS, Gilbert1991, GHS08}.
\begin{definition}\label{CliffAlg}
	The Clifford algebra associated with $\mathbb{R}^{n}$, endowed with the usual Euclidean metric, is the minimal enlargement of $\mathbb{R}^{n}$ to a real unitary, associative algebra $\mathcal{C}\ell(n)$ and satisfies 
	\begin{eqnarray*}
		x^{2} = -\arrowvert x\arrowvert^{2},   
	\end{eqnarray*}
	for   any  $x \in \mathbb{R}^{n}$.
\end{definition}
We denote the standard orthonormal basis of $\mathbb{R}^{n}$ by $\{e_{j}\}_{j = 1}^{n}$ under the multiplication rules
\begin{equation*}
	e_{i}e_{j} + e_{j}e_{i} = -2\delta_{ij},
\end{equation*}
where $\delta_{ij}$ stands for the Kronecker delta.

Each element $a \in \mathcal{C}\ell(n)$ has the form $a = \sum\limits_{A \subseteq N}a_{A}e_{A}$, $N = \{1, \ldots, n \}$, $a_{A} \in \mathbb{R}$ where $e_{\emptyset} = e_{0} = 1$, $e_{\{ j\}} = e_{j}$ and $e_{A} = e_{\beta_{1}}\cdots e_{\beta_{k}}$ for $A = \{\beta_{1}, \ldots, \beta_{k} \}$ where $\beta_{j} \in \{1, \ldots, n \}$ and $\beta_{1} < \ldots < \beta_{k}$.

The Clifford algebra $\mathcal{C}\ell(n)$ endowed with the norm $\arrowvert a\arrowvert = \left(\sum_{A}a_{A}^{2}\right)^{\frac{1}{2}}$ becomes a Euclidean space. The anti-involution $a\mapsto\overline{a}:= \sum_{A}a_{A}\overline{e}_{A}$ for which 
\begin{equation*}
	\overline{e}_{A} := (-1)^{k}e_{\beta_{k}} \cdots e_{\beta_{2}}e_{\beta_{1}}.
\end{equation*}
define the conjugation in $\mathcal{C}\ell(n)$.

We define $\mathcal{C}\ell(n)^{(k)} = span_{\mathbb{R}}(e_{A}: |A| = k)$. The subspace $\mathcal{C}\ell(n)^{(k)}$ of $\mathcal{C}\ell(n)$ is named the $k$-vectors space and 
\begin{equation*}
	\mathcal{C}\ell(n) = \bigoplus_{k = 0}^{n}\mathcal{C}\ell(n)^{(k)}.
\end{equation*}
The space of paravectors $\mathcal{C}\ell(n)^{(0)}\oplus\mathcal{C}\ell(n)^{(1)}$ is an important subspace of the real Clifford algebra $\mathcal{C}\ell(n)$. We should note that every $x = (x^{0}, x^{1}, \ldots, x^{n}) \in \mathbb{R}^{n + 1}$ will be identified with
\begin{equation*}
	x = x^{0} + \sum_{j = 1}^{n}x^{j}e_{j} \in \mathcal{C}\ell(n)^{(0)}\oplus\mathcal{C}\ell(n)^{(1)}.
\end{equation*}
Analogously, $\ux = (x^{1}, \ldots, x^{n}) \in \mathbb{R}^{n}$ will be identified with
\begin{equation*}
	\ux = \sum_{j = 1}^{n}x^{j}e_{j} \in  \mathcal{C}\ell(n)^{(1)}.
\end{equation*}
Let us highlight the fact that for each $i_{0} \in \left\lbrace 1, 2, \dots, n \right\rbrace$ we have that $\mathcal{C}\ell(n) = \mathcal{C}\ell(n)^{+}\oplus e_{i_{0}}\mathcal{C}\ell(n)^{+}$, where 
\[\mathcal{C}\ell(n)^{+}:=\bigoplus_{k-even}\mathcal{C}\ell(n)^{(k)}.\]
Hence, for every $a \in \mathcal{C}\ell(n)$ the following decomposition holds
\begin{equation}\label{eo}
	a = a_{0} + e_{i_{0}}a_{1},
\end{equation}
where $a_{0}, a_{1}$ will be respectively referred to as its even and odd parts.

The functions $u$ to be considered are defined in $\Omega\subset\mathbb{R}^{n + 1}$ (respectively $\Omega\subset\mathbb{R}^{n}$) and take values in $\mathcal{C}\ell(n)$. They have the representation
\begin{equation*}
	u(x) = \sum_{A}u_{A}(x)e_{A},
\end{equation*}
here $u_{A}$ are $\mathbb{R}$-valued functions. The study of these functions will be called paravectorial (respectively vectorial) Clifford analysis. If each component $u_{A}$ belongs to some classical class of functions on $\Omega$ we say that $u$ belongs to that class.

The Cauchy-Riemann and Dirac operators are respectively defined by
\begin{equation*}
	\mathcal{D}_{n} := \sum^{n}_{j = 0}e_{j}\dfrac{\partial}{\partial x_{j}}.
\end{equation*}
\begin{equation*}
	\partial_{n} := \sum^{n}_{j = 1}e_{j}\dfrac{\partial}{\partial x^{j}}.
\end{equation*}
Let $\Omega\subset\mathbb{R}^{n + 1}$ be an open set and $u \in C^{1}(\Omega)$ then $u$ will be called left (respectively right) monogenic in $\Omega$ if $\mathcal{D}_{n}u = 0$ (respectively $u\mathcal{D}_{n} = 0$) in $\Omega$. Analogously, in the vectorial approach, the same notation will be used if $\partial_{n}u = 0$ (resp. $u\partial_{n} = 0$).

The fundamental solution of the Cauchy-Riemann and Dirac operators are respectively given by 
\begin{equation*}\label{Ch1EqFundSolCauchRiem}
	\begin{array}{cc}
		E_{n}(x) = \dfrac{1}{\sigma_{n + 1}}\dfrac{\overline{x}}{|x|^{n + 1}}, &  x\in\mathbb{R}^{n + 1}\setminus \{0\}
	\end{array} 
\end{equation*}
\begin{eqnarray*}
	\begin{array}{cc}
		\vartheta_{n}(\ux) = \dfrac{1}{\sigma_{n}}\dfrac{\overline{\ux}}{|\ux|^{n}}, &  \ux\in\mathbb{R}^{n}\setminus \{0\}
	\end{array} 
\end{eqnarray*}
where $\sigma_{n}$ is the area of the unit sphere in $\mathbb{R}^{n}$. 

The theory of monogenic functions with values in Clifford algebras is a natural generalization of the theory of holomorphic functions in one variable complex analysis to a higher dimensional Euclidean space. 

We will take up the Teodorescu transform, see \cite{GHS08} for more details.
\begin{definition}\label{Ch1DefTeodTransfRig}
	Let $\Omega\subset \mathbb{R}^{n+1}$ be a domain and let $u \in C^{1}(\overline{\Omega})$, the operators defined by 
	\begin{eqnarray*}
		(T^{l}_{\Omega}u)(x) = -\int\limits_{\Omega}E_{n}(y-x)u(y)dV(y), & x \in \mathbb{R}^{n + 1}, 
	\end{eqnarray*}
	\begin{eqnarray*}
		(T^{r}_{\Omega}u)(x) = -\int\limits_{\Omega}u(y)E_{n}(y-x)dV(y), & x \in \mathbb{R}^{n + 1}, 
	\end{eqnarray*}
	where $dV(y)$ is the volume element, are called the left and right Teodorescu transform respectively. 
\end{definition} 
The corresponding Teodorescu transform in the vectorial Clifford analysis setting occurs analogously with $E_{n}(x)$ replaced by $\vartheta_{n}(\ux)$.
\subsection{Fractal Dimensions and Characteristics}\label{SubSecFracDim}
In order to make this paper self-contained, we present some basic notions about fractal dimensions and characteristics. For a fuller treatment of this topic, we refer the reader to \cite{Fal, Mandelbrot, Matt}.

Let $\textbf{E}$ be a nonempty set in $\mathbb{R}^{n}$. For each $\delta > 0$ and $s \geq 0$, the value $\mathcal{H}_{\delta}^{s}(\textbf{E})$  is defined as, 
\begin{equation*}
	\mathcal{H}_{\delta}^{s}(\textbf{E}) := \inf\left\lbrace \sum_{i = 1}^{\infty}\diam(U_{i})^{s}: \{U_{i}\} \ is \ a  \ \delta-covering \ of \ \textbf{E}\right\rbrace ,
\end{equation*} 
where $\diam(U)$ stands for the diameter of the set $U$. Here, the infimum is understood over all countable $\delta$-coverings ${U_{i}}$ of $\textbf{E}$ with open or closed balls. 
\begin{definition}
	The $s$-dimensional Hausdorff measure is defined by	the limit
	\begin{equation*}
		\mathcal{H}^{s}(\textbf{E}) := \lim _{\delta\rightarrow 0}\mathcal{H}_{\delta}^{s}(\textbf{E}).
	\end{equation*} 
\end{definition}
Then, the Hausdorff dimension can be introduced. 
\begin{definition}\label{DefHausDim}
	The Hausdorff dimension of \textbf{E} is defined as
	\begin{equation*}
		\dim _{H}\textbf{E} := \inf\left\lbrace s \geq 0: \mathcal{H}^{s}(\textbf{E}) = 0\right\rbrace  = \sup\left\lbrace s \geq 0: \mathcal{H}^{s}(\textbf{E}) = \infty\right\rbrace.
	\end{equation*}   
\end{definition}
The following definition of fractal set was given in \cite{Mandelbrot}, where the term fractal was coined.
\begin{definition}\label{DefFractMan}
	If an arbitrary set $\textbf{E} \subset \mathbb{R}^{n+1}$ with topological dimension $n$ has $\dim_{H}\textbf{E} > n$,  then $\textbf{E}$ is called a fractal set in the sense of Mandelbrot.
\end{definition}
The following new metric characteristics of a fractal set in $\mathbb{R}^{n + 1}$ are mainly included to keep the exposition as self-contained as possible. These can be found in \cite{TamayoCastro2022}.

Let $\mathcal{S}$ be a topologically compact hypersurface in $\mathbb{R}^{n +1}$, which bounds a Jordan domain $\Omega^{+}$. We write $\Omega^{-}$ for the unbounded complement. It is assumed $\mathcal{S}$ to be fractal.

Let $D \subset \mathbb{R}^{n + 1}$ be a bounded set, which does not touch the hypersurface $\mathcal{S}$. We will consider the integral
\begin{displaymath}
	I_{p}(D) = \int\limits_{D}\dfrac{dV(x)}{\dist^{p}(x, \mathcal{S})}.
\end{displaymath}

\begin{definition}\label{DefMarcExp}
	Let $\mathcal{S}$ be a topologically compact hypersurface which is the boundary of a Jordan domain in $\mathbb{R}^{n +1}$. We define the inner and outer Marcinkiewicz exponent of $\mathcal{S}$, respectively, as
	\begin{displaymath}
		\begin{array}{cc}
			\mathfrak{m}^{+}(\mathcal{S}) = \sup \{p: I_{p}(\Omega^{+}) < \infty\},  & \mathfrak{m}^{-}(\mathcal{S}) = \sup\{p: I_{p}(\Omega^{*}) < \infty\}, 
		\end{array} 
	\end{displaymath}
	with the notation $\Omega^{*} := \Omega^{-}\bigcap\{x: \arrowvert x\arrowvert < r\}$, where $r$ is chosen in a way that $\mathcal{S}$ is completely contained inside the ball of radius $r$. 	
	
	The (absolute) Marcinkiewicz exponent of $\mathcal{S}$ is defined by,
	\begin{displaymath}
		\mathfrak{m}(\mathcal{S}) = \max\{\mathfrak{m}^{+}(\mathcal{S}), \mathfrak{m}^{-}(\mathcal{S})\}.
	\end{displaymath} 
\end{definition}
We remark that the value of $\mathfrak{m}^{-}(\mathcal{S})$ is independent of the choice of the radius $r$ in the construction of $\Omega^{*}$, due to the fact that only the points closest to $\mathcal{S}$ have any influence the convergence of the integral $I_{p}(D)$.

\subsection{H\"older continuous functions and Whitney extension Theorem}
The data functions of the Riemann boundary value problems to be studied throughout the paper belong to the H\"older class of continuous functions, see \cite{St}. 
\begin{definition}\label{Ch1DefLip}
	Let be $\E \subset \mathbb{R}^{m}$ and $0 < \nu \leq 1$, a function $f: \E \mapsto \mathbb{R}$ satisfying 
	\begin{eqnarray*}\label{HoldCond}
		|f(x) - f(y)| \leq M|x - y|^{\nu}; &  x, y \in \E,
	\end{eqnarray*}
	is called H\"older continuous function with exponent $\nu$. The space of all these functions will be denoted by $C^{0,\nu}(\E)$.
\end{definition}
The Whitney type extension theorem for Clifford valued functions was presented in \cite{ABT2007}. It was based on the result \cite[pp 174]{St} for real-valued functions, originally stated by Hassler Whitney. 
\begin{theorem}\label{ExWhitney}
	Let $\textbf{E} \subset \mathbb{R}^{n+1}$ be a compact set and let $u \in C^{0,\nu}(\E)$, with $0 < \nu \leq 1$. Then, there exists a function $\widetilde{u} \in C^{0,\nu}(\mathbb{R}^{n+1})$, named Whitney extension operator of $u$, that satisfies\\
	(i) $\widetilde{u}\arrowvert_{E} = u$,\\
	(ii) $\widetilde{u} \in C^{\infty}(\mathbb{R}^{n+1}\setminus \textbf{E})$,\\
	(iii) $\arrowvert \mathcal{D}\widetilde{u}(x) \arrowvert \leq C\dist(x, \textbf{E})^{\nu - 1}$ for $x \in \mathbb{R}^{n+1}\setminus \textbf{E}$.
\end{theorem}
We refer the reader to \cite{APB2007} for the proof of the following Dolzhenko type theorem.
\begin{theorem}\label{CorrDolz}
	Let $\Omega$ be a domain in $\mathbb{R}^{n+1}$ and $\textbf{E} \subset \Omega$ be a compact set. Let $\mathcal{H}^{n + \nu}(\textbf{E}) = 0$ where $0 < \nu \leq 1$. If $u \in C^{0,\nu}(\Omega)$, and it is monogenic in $\Omega\setminus\textbf{E}$, then $u$ is also monogenic in $\Omega$.   
\end{theorem}  

\section{Reduction Procedure for the Riemann Boundary Value Problem}\label{SecReductionProcedure}
The reduction procedure for the Riemann boundary value problem relative to the elements $e_1$ and $e_n$ was developed in \cite{TAB21}. In this section, we extend that idea in order to present a reduction procedure relative to any fixed element $e_{i_{0}}$, where $i_{0} \in \left\lbrace 1, 2, \dots, n \right\rbrace$.

Let $\mathcal{S}$ be a topologically compact surface that is the boundary of a Jordan domain $\Omega \subset \mathbb{R}^{n}$. Let us introduce the temporary notation $\Omega^+:=\Omega$, and $\Omega^-:=\R^n\setminus\{\Omega\cup \mathcal{S}\}$. The Riemann boundary value problem in vectorial Clifford analysis can be presented as follows: Let two $\mathcal{C}\ell(n)-$valued functions $G, g \in C^{0,\nu}(\mathcal{S})$. To find a function $\Phi$ monogenic on $\R^n\setminus \mathcal{S}$ continuously extendable from $\Omega^{\pm}$ to $\mathcal{S}$ such that the following condition of their boundary values $\Phi^{\pm}$ on $\mathcal{S}$ satisfy
\begin{eqnarray}\label{ProbOrig}
	\Phi^{+}(\ux) - G(\ux)\Phi^{-}(\ux) = g(\ux), & \ux\in \mathcal{S}.
\end{eqnarray}
with $G(\ux) \neq 0$. In general, we assume $\Phi(\infty) = c$, with $c$ being a constant.

Let us consider the algebra isomorphism
\begin{equation*}
	\begin{array}{cccc}
		\beta_{i_{0}}: & \mathcal{C}\ell(n)^{+}   & \rightarrow &  \mathcal{C}\ell(n-1)\\
		&     e_{i_{0}}e_{i}      & \rightarrow & \left\lbrace \begin{array}{cc}
			e_{i} & i < i_{0}	\\
			e_{i - 1} & i > i_{0}	
		\end{array} \right. 
	\end{array} ,
\end{equation*}
as well as the isomorphism of vector space 
\begin{equation*}
	\begin{array}{cccc}
		\alpha_{i_{0}}: & \mathcal{C}\ell(n)^{(1)}  & \rightarrow & \mathcal{C}\ell(n-1)^{(0)}\oplus\mathcal{C}\ell(n-1)^{(1)}   \\ 
		& \sum_{i = 1}^{n}x^{i}e_{i} & \rightarrow  & x^{i_{0}} + \sum_{i = 1}^{i_{0}}x^{i}e_{i} + \sum_{i = i_{0} + 1}^{n}x^{i - 1}e_{i - 1}.  
	\end{array} 
\end{equation*}
If $u: \mathcal{C}\ell(n)^{(1)} \rightarrow \mathcal{C}\ell(n)$ the following decomposition holds
\begin{equation*}
	u = u_{0} + e_{i_{0}}u_{1},
\end{equation*}
where $u_{0}, u_{1}: \mathcal{C}\ell(n)^{(1)} \rightarrow \mathcal{C}\ell(n)^{+}$ are its even and odd parts, respectively.

Hereafter, for a given function
\begin{equation*}
	\begin{array}{cccc}
		u: & \mathcal{C}\ell(n)^{(1)}  & \rightarrow & \mathcal{C}\ell(n)^{+},   \\ 
	\end{array} 
\end{equation*}
we define
\begin{equation*}
	\begin{array}{cccc}
		\widehat{u} : & \mathcal{C}\ell(n-1)^{(0)}\oplus\mathcal{C}\ell(n-1)^{(1)}   & \rightarrow & \mathcal{C}\ell(n-1),   \\ 
	\end{array} 
\end{equation*}
by $\widehat{u}(y) := \beta_{i_{0}} \circ u \circ \alpha_{i_{0}} (\ux) = \beta_{i_{0}}(u(\alpha_{i_{0}}(\ux)))$, and $y = \alpha_{i_{0}}(\ux)$.

Using the decomposition \eqref{eo} to the functions involved in (\ref{ProbOrig}), we arrive at the system
\begin{equation}\label{System}
	\Big\lbrace \begin{array}{c}
		\Phi_{0}^{+}(\ux) -(G_{0}(\ux)\Phi_{0}^{-}(\ux) - G^{*}_{1}(\ux)\Phi_{1}^{-}(\ux)) = g_{0}(\ux), \\
		\Phi_{1}^{+}(\ux) - (G_{1}(\ux)\Phi_{0}^{-}(\ux) + G^{*}_{0}(\ux)\Phi_{1}^{-}(\ux))= g_{1}(\ux).
	\end{array}
\end{equation}
where $G^{*}_{j}(\ux) = -e_{i_{0}}G_{j}(\ux)e_{i_{0}} \in \mathcal{C}\ell(n)^{+}, j = 0, 1$.

Making use of the transformations $\beta_{i_{0}}$ and $\alpha_{i_{0}}$ the system (\ref{System}) becomes 
\begin{equation}\label{RBVP}
	\begin{array}{ccccc}
		\begin{pmatrix}
			\widehat{\Phi}_{0}^{+}(y)\\ 
			\widehat{\Phi}_{1}^{+}(y)
		\end{pmatrix} & - & \begin{pmatrix}
			\widehat{G}_{0}(y) & -\widehat{G}^{*}_{1}(y) \\ 
			\widehat{G}_{1}(y) & \widehat{G}^{*}_{0}(y)
		\end{pmatrix}\begin{pmatrix}
			\widehat{\Phi}_{0}^{-}(y)\\ 
			\widehat{\Phi}_{1}^{-}(y)
		\end{pmatrix}  & = & \begin{pmatrix}
			\widehat{g}_{0}(y)\\ 
			\widehat{g}_{1}(y)
		\end{pmatrix}.
	\end{array} 
\end{equation}
We shall rewrite the Dirac operator $\partial_{n}$ in the following form
\begin{equation*}
	\partial_{n} := \sum_{i = 1}^{n}e_{i}\dfrac{\partial}{\partial x^{i}} = e_{i_{0}}(\dfrac{\partial}{\partial x^{i_{0}}} - \sum_{i = 1}^{i_{0}-1}e_{i_{0}}e_{i}\dfrac{\partial}{\partial x^{i}} -\sum_{i = i_{0}+1}^{n}e_{i_{0}}e_{i}\dfrac{\partial}{\partial x^{i}}) =: e_{i_{0}}\overline{\mathcal{D}_{n - 1}^{'}},
\end{equation*}
\begin{equation*}
	\partial_{n} := \sum_{i = 1}^{n}e_{i}\dfrac{\partial}{\partial x^{i}} = (\dfrac{\partial}{\partial x^{i_{0}}} + \sum_{i = 1}^{i_{0}-1}e_{i_{0}}e_{i}\dfrac{\partial}{\partial x^{i}} + \sum_{i = i_{0}+1}^{n}e_{i_{0}}e_{i}\dfrac{\partial}{\partial x^{i}})e_{i_{0}} =: \mathcal{D}_{n - 1}^{'} e_{i_{0}}.
\end{equation*}
Observe that
\begin{equation*}
	\partial_{n}\Phi = e_{i_{0}}\overline{\mathcal{D}_{n - 1}^{'}}(\Phi_{0}) - \mathcal{D}_{n - 1}^{'}(\Phi_{1}) = - \mathcal{D}_{n - 1}^{'}(\Phi_{1}) +  e_{i_{0}}\overline{\mathcal{D}_{n - 1}^{'}}(\Phi_{0}).
\end{equation*}
Therefore, we get the following equivalence 
\begin{equation}\label{SistMon}
	\begin{array}{ccc}
		\partial_{n}\Phi = 0 & \Leftrightarrow  & \left\lbrace  \begin{array}{c}
			\mathcal{D}_{n - 1}^{'}(\Phi_{1}) = 0\\ 
			\overline{\mathcal{D}_{n - 1}^{'}}(\Phi_{0}) = 0
		\end{array} \right. .
	\end{array} 
\end{equation}
Additionally, we have
\begin{equation*}
	\mathcal{D}_{n - 1} = \beta_{i_{0}}(\mathcal{D}_{n - 1}^{'}) = \dfrac{\partial}{\partial x^{i_{0}}} +  \sum_{i = 1}^{i_{0}-1}e_{i}\dfrac{\partial}{\partial x^{i}} + \sum_{i = i_{0}}^{n -1}e_{i}\dfrac{\partial}{\partial x^{i + 1}} = \partial_{i_{0}} + \partial_{n - 1},
\end{equation*}
and
\begin{equation*}
	\overline{\mathcal{D}_{n - 1}} = \beta_{i_{0}}(\overline{\mathcal{D}_{n - 1}^{'}}) = \dfrac{\partial}{\partial x^{i_{0}}} - \sum_{i = 1}^{i_{0}-1}e_{i}\dfrac{\partial}{\partial x^{i}} - \sum_{i = i_{0}}^{n -1}e_{i}\dfrac{\partial}{\partial x^{i + 1}} = \partial_{i_{0}} - \partial_{n - 1}.
\end{equation*}
Hence, the system ($\ref{SistMon}$) becomes
\begin{equation*}
	\left\lbrace  \begin{array}{c}
		\mathcal{D}_{n - 1}(\widehat{\Phi}_{1}) = 0\\ 
		\overline{\mathcal{D}_{n - 1}}(\widehat{\Phi}_{0}) = 0.
	\end{array} \right. 
\end{equation*}
Thus, the even and odd parts of $\widehat{\Phi}$ are antimonogenic and monogenic functions, respectively.  

We conclude this section with the following theorem.
\begin{theorem}
	A function $\Phi(\ux)$ is monogenic in the vectorial sense if and only if its even and odd parts are, through isomorphism, antimonogenic and monogenic, respectively in the paravectorial sense, and is decomposed in the form
	\begin{equation*}
		\Phi(\ux) = \beta_{i_{0}}^{-1}(\widehat{\Phi}_{0}(\ux)) + e_{i_{0}}\beta_{i_{0}}^{-1}(\widehat{\Phi}_{1}(\ux)),
	\end{equation*}
	where $\ux = \alpha_{i_{0}}^{-1}(y).$
\end{theorem}

\section{Applications on fractal boundaries}\label{SecApplicationsFractal}
In this section, we will apply the method developed in Section \ref{SecReductionProcedure} to domains with fractal boundaries.

Due to the fact that $\widehat{\Phi}_{0}(y)$ is a left antimonogenic function, then $\Upsilon_{0}(y):= \overline{\widehat{\Phi}_{0}(y)}$ is a right monogenic function. For similarity we set $\Upsilon_{1} = \widehat{\Phi}_{1}$. Thus, problem (\ref{RBVP}) reduces to find 
\begin{equation*}
	\begin{pmatrix}
		\Upsilon_{0}(y)\\ 
		\Upsilon_{1}(y)
	\end{pmatrix},
\end{equation*}
such that on $\mathbb{R}^{n}\setminus \mathcal{S}$
\begin{equation*}
	\left\lbrace  \begin{array}{c}
		(\Upsilon_{0})\mathcal{D}_{n - 1} = 0\\ 
		\mathcal{D}_{n - 1}(\Upsilon_{1}) = 0,
	\end{array} \right. 
\end{equation*}
while on $\mathcal{S}$ the boundary condition
\begin{equation}\label{CondComp2}
	\begin{array}{ccccc}
		\begin{pmatrix}
			\overline{\Upsilon_{0}^{+}(y)}\\ 
			\Upsilon_{1}^{+}(y)
		\end{pmatrix} & - & \begin{pmatrix}
			\widehat{G}_{0}(y) & -\widehat{G}^{*}_{1}(y) \\ 
			\widehat{G}_{1}(y) & \widehat{G}^{*}_{0}(y)
		\end{pmatrix}\begin{pmatrix}
			\overline{\Upsilon_{0}^{-}(y)}\\ 
			\Upsilon_{1}^{-}(y)
		\end{pmatrix}  & = & \begin{pmatrix}
			\widehat{g}_{0}(y)\\ 
			\widehat{g}_{1}(y)
		\end{pmatrix}
	\end{array} 
\end{equation}
is satisfied.

Whenever this problem is solvable then so it is ($\ref{ProbOrig}$) and the explicit solution is given by
\begin{equation*}
	\Phi(\ux) = \beta_{i_{0}}^{-1}\circ\overline{\Upsilon_{0}}\circ\alpha_{i_{0}}^{-1}(y) + e_{i_{0}}\beta_{i_{0}}^{-1}\circ\Upsilon_{1}\circ\alpha_{i_{0}}^{-1}(y) = 
\end{equation*}
\begin{equation}\label{Ch3EqSolCuater2nd}
	= \beta_{i_{0}}^{-1}(\overline{\Upsilon_{0}}(\ux)) + e_{i_{0}}\beta_{i_{0}}^{-1}(\Upsilon_{1}(\ux)),
\end{equation}
where
\begin{equation*}
	\ux = \alpha_{i_{0}}^{-1}(y).
\end{equation*}
\subsection{Conditions in the vectorial approach through the paravectorial}
The following theorems were proved in \cite{TamayoCastro2022} in the context of the paravectorial Clifford analysis, i.e. when studying functions $f: \mathbb{R}^{n+1}\rightarrow \mathcal{C}\ell(n)$.
\begin{theorem}\label{TheoSolvCond}
	Let $\mathcal{S}$ be a topologically compact surface which is the boundary of a Jordan domain in $\mathbb{R}^{n +1}$,  and let $f \in C^{0, \nu}(\mathcal{S})$. If 
	\begin{equation*}
		\nu > 1 - \dfrac{\mathfrak{m}(\mathcal{S})}{n + 1},
	\end{equation*}
	and the scalar additive Jump problem is solvable.
\end{theorem}

\begin{theorem}\label{TheoUniMar}
	Let $\mathcal{S}$ be a topologically compact surface which is the boundary of a Jordan domain in $\mathbb{R}^{n + 1}$, and let $f \in C^{0, \nu}(\mathcal{S})$, with $\nu > 1 - \dfrac{\mathfrak{m}(\mathcal{S})}{n + 1}$ and 
	\begin{equation}\label{EqUniCndMar}
		\dim_{H}\mathcal{S} - n < \mu < 1 - \dfrac{(n + 1)(1 - \nu)}{\mathfrak{m}(\mathcal{S})}.
	\end{equation}
	Then the solution to the scalar additive Jump problem is unique in the classes $C^{0, \mu}(\overline{G^{+}})$ and $C^{0, \mu}(\overline{G^{-}})$.
\end{theorem}
\begin{remark}\label{RemarkRightMonValid}
	It is easy to check that Theorems \ref{TheoSolvCond} and \ref{TheoUniMar} remain valid when the scalar additive Jump problem is stated for the right monogenic functions.	Additionally, they are also valid in the context of vectorial Clifford analysis with the corresponding change of the value $n+1$ for $n$ in the inequalities.  As a matter of fact, those follow with analogous reasoning to that in \cite{TamayoCastro2022} by using the properties of the Teodorescu transform written in the vectorial sense, see  \cite{APB2007, GS1997}. 
\end{remark}
We shall use the reduction procedure for the RBVP to obtain unique solvability conditions in the vectorial approach through the paravectorial one. This method has proved to be more effective in the sense that we get more different expressions for the solutions.  

In the case of the scalar additive Jump problem $G(\ux) \equiv 1$, (\ref{ProbOrig}) becomes
\begin{eqnarray}\label{Ch3EqJumPro}
	\Phi^{+}(\ux) - \Phi^{-}(\ux) = g(\ux), & \ux\in \mathcal{S}.
\end{eqnarray}
Consequently, ($\ref{CondComp2}$) turns into 
\begin{equation*}
	\Big\lbrace \begin{array}{c}
		\overline{\Upsilon_{0}^{+}(\ux)} - \overline{\Upsilon_{0}^{-}(\ux)} = \widehat{g}_{0}(\ux) \\
		\Upsilon_{1}^{+}(\ux) -\Upsilon_{1}^{-}(\ux) = \widehat{g}_{1}(\ux),
	\end{array} 
\end{equation*}
which is equivalent to
\begin{equation}\label{Ch3EqEquivJumPro}
	\Big\lbrace \begin{array}{c}
		\Upsilon_{0}^{+}(\ux) - \Upsilon_{0}^{-}(\ux) = \overline{\widehat{g}_{0}(\ux)} \\
		\Upsilon_{1}^{+}(\ux) -\Upsilon_{1}^{-}(\ux) = \widehat{g}_{1}(\ux).
	\end{array} 
\end{equation}
These are two independent Jump problems in paravectorial Clifford analysis, the first one for right monogenic functions and the second for left monogenic ones. Moreover, problem (\ref{Ch3EqJumPro}) is solvable if and only if both problems in the system (\ref{Ch3EqEquivJumPro}) are solvable. 

If $g$ belongs to a H\"older class then its even and odd parts also are in H\"older class, which is enough for applications to domains with smooth boundaries. However, it is not the case for fractal domains, where to show a relation between the H\"older exponent of $g$ and those of its even and odd parts is required. These relations are established by our next theorem.

We write $\alpha$ for the maximum H\"older exponent of $g$ on a set $\textbf{E}\subset \mathbb{R}^{n + 1}$ requiring that $g \in  C^{0,\nu}(\textbf{E})$ for every $\nu \leq \alpha$, while $g \notin  C^{0,\mu}(\textbf{E})$ for $\mu > \alpha$, see \cite{Fiorenza2016}. However, as was also shown in \cite{Fiorenza2016}, the maximum H\"older exponent do not always exists, i.e. in some cases $g \in  C^{0,\nu}(\textbf{E})$ for every $\nu < \alpha$, while $g \notin  C^{0,\alpha}(\textbf{E})$. In order to contain both cases, we will introduce the concept of supremum of the H\"older exponents. 
\begin{definition}
	We say that $\alpha$ is the supremum of the H\"older exponents on $\textbf{E}$ if $g \in  C^{0,\nu}(\textbf{E})$ for every $\nu < \alpha$, while $g \notin  C^{0,\mu}(\textbf{E})$ for $\mu > \alpha$.
\end{definition} 
\begin{theorem}\label{Ch3TheoHoldPart}
	A function $g(\ux) \in  C^{0,\nu}(\mathcal{S})$, with $\alpha$ the supremum of the H\"older exponents on $\mathcal{S}$, if and only if its even and odd part $\widehat{g}_{0}(y) \in C^{0,\nu_{0}}(\mathcal{S})$ and $\widehat{g}_{1}(y) \in C^{0,\nu_{1}}(\mathcal{S})$, with $\alpha_{0}$ and $\alpha_{1}$ the supremum of the H\"older exponents on $\mathcal{S}$, respectively, and $\alpha = \min \left\lbrace \alpha_{0}, \alpha_{1}\right\rbrace$. 
\end{theorem}
\begin{proof}
	First we shall show that if $\widehat{g}_{0}(\ux) \in C^{0,\nu_{0}}(\mathcal{S})$ and $\widehat{g}_{1}(\ux) \in C^{0,\nu_{1}}(\mathcal{S})$, with $\alpha_{0}$ and $\alpha_{1}$ the supremum of the H\"older exponents, respectively, and $\alpha = \min \left\lbrace \alpha_{0}, \alpha_{1}\right\rbrace$ then $g(\ux) \in C^{0,\nu}(\mathcal{S})$ with $\alpha$ the supremum of the H\"older exponents.
	Indeed, we have,
	\begin{equation}\label{Ch3EqLemHolDerivation}
		\begin{array}{cc}
			\rvert g(\ux) - g(\uy)\lvert = \rvert g_{0}(\ux) + e_{i_{0}}g_{1}(\ux) - g_{0}(\uy) + e_{i_{0}}g_{1}(\uy) \lvert = & \\ 
			\rvert \left[ g_{0}(\ux)-g_{0}(\uy)\right]  + e_{i_{0}}\left[ g_{1}(\ux) - g_{1}(\uy)\right]\lvert \leq & \\
			\rvert \left[ g_{0}(\ux)-g_{0}(\uy)\right]\lvert  + \rvert\left[ g_{1}(\ux) - g_{1}(\uy)\right]\lvert \leq & \\
			C_{0}\rvert\ux - \uy\lvert^{\nu_{0}} + C_{1}\rvert\ux - \uy\lvert^{\nu_{1}},  & \nu_{0} < \alpha_{0} \,\,\,\, \mathrm{and} \,\,\,\, \nu_{1} < \alpha_{1}, \\		
		\end{array}
	\end{equation}
	thus
	\begin{equation*}
		\begin{array}{cc} 
			\rvert g(\ux) - g(\uy)\lvert \leq C\rvert\ux - \uy\lvert^{\nu} & \nu < \alpha.		
		\end{array}
	\end{equation*}
	We will suppose that $g(\ux) \in C^{0,\mu}(\mathcal{S})$ for some $\mu > \alpha = \min \left\lbrace \alpha_{0}, \alpha_{1}\right\rbrace$. Without loss of generality, we may assume that $\alpha_{0} \leq \alpha_{1}$. We can proceed analogously if the opposite is supposed. Thus we obtain,
	\begin{eqnarray*}
		C\rvert\ux - \uy\lvert^{\mu} \geq \rvert g(\ux) - g(\uy)\lvert = \left(\sum_{A\subset N}[g_{A}(\ux) - g_{A}(\uy)]^{2} \right)^{\frac{1}{2}} \geq &\\
		\left(\sum_{\rvert A\lvert \, even}[g_{A}(\ux) - g_{A}(\uy)]^{2} \right)^{\frac{1}{2}} = \rvert g_{0}(\ux) - g_{0}(\uy)\lvert, & \mu > \alpha_{0}.
	\end{eqnarray*}
	That contradicts the fact that $\alpha_{0}$ is the supremum of the H\"older exponents of $g_{0}$. Therefore, $\alpha = \min \left\lbrace \alpha_{0}, \alpha_{1}\right\rbrace$ is the supremum of the H\"older exponents of $g$. \\
	Let us now show that if $g(\ux) \in C^{0,\nu}(\mathcal{S})$, with $\alpha$ the supremum of the H\"older exponents, then $\widehat{g}_{0}(\ux) \in C^{0,\nu_{0}}(\mathcal{S})$ and $\widehat{g}_{1}(\ux) \in C^{0,\nu_{1}}(\mathcal{S})$, with $\alpha_{0}$ and $\alpha_{1}$ the supremum of the H\"older exponents, respectively, and $\alpha = \min \left\lbrace \alpha_{0}, \alpha_{1}\right\rbrace$. Then we have,
	\begin{eqnarray*}
		C\rvert\ux - \uy\lvert^{\nu} \geq \rvert g(\ux) - g(\uy)\lvert \geq \rvert g_{0}(\ux) - g_{0}(\uy)\lvert, & \nu > \alpha.\\
	\end{eqnarray*}
	We have proved that if $g(\ux) \in C^{0,\nu}(\mathcal{S})$ then $\widehat{g}_{0}(\ux) \in C^{0,\nu}(\mathcal{S})$, both with $\nu > \alpha$. However, there could also happen that $\widehat{g}_{0}(\ux) \in C^{0,\nu}(\mathcal{S}) \subset C^{0,\nu_{0}}(\mathcal{S})$ with $\nu \leq \nu_{0}$ and $\alpha < \alpha_{0}$. An analogous reasoning can be done with $\widehat{g}_{1}(\ux)$, $\nu \leq \nu_{1}$ and $\alpha < \alpha_{1}$. We are going to show that $\alpha = \min \left\lbrace \alpha_{0}, \alpha_{1}\right\rbrace$. Assuming that $\alpha < \min \left\lbrace \alpha_{0}, \alpha_{1}\right\rbrace$. Using a similar analysis than in (\ref{Ch3EqLemHolDerivation}) we get that 
	\begin{eqnarray*}
		\rvert g(\ux) - g(\uy)\lvert \leq C\rvert\ux - \uy\lvert^{\nu}, & \nu < \min \left\lbrace \alpha_{0}, \alpha_{1}\right\rbrace,
	\end{eqnarray*}
	that contradict the fact that the supremum of the H\"older exponents of the function $g$, therefore, the supposition is false and $\alpha = \min \left\lbrace \alpha_{0}, \alpha_{1}\right\rbrace$. This completes the proof.
\end{proof}
Now, we are going to prove a sufficient solvability condition for the Jump problem in the vectorial approach.
\begin{theorem}\label{Ch3TheoSolvCond}
	Let $\mathcal{S}$ be a topologically compact surface which is the boundary of a Jordan domain in $\mathbb{R}^{n}$,  and let $f \in C^{0,\nu}(\mathcal{S})$. If 
	\begin{equation}\label{Ch3EqCndMar}
		\nu > 1 - \dfrac{\mathfrak{m}(\mathcal{S})}{n},
	\end{equation}
	then the scalar additive Jump Problem (\ref{Ch3EqJumPro}) is solvable.
\end{theorem}
\begin{proof}
	If $\widehat{g}_{0}(\ux) \in C^{0,\nu_{0}}(\mathcal{S})$ and $\widehat{g}_{1}(\ux) \in C^{0,\nu_{1}}(\mathcal{S})$ then from Theorem \ref{TheoSolvCond} and Remark \ref{RemarkRightMonValid} we know that the system (\ref{Ch3EqEquivJumPro}) is solvable if
	\begin{equation}\label{Ch3ConEquivJumPro1}
		\nu_{0} > 1 - \dfrac{\mathfrak{m}(\mathcal{S})}{(n - 1) + 1} = 1 - \dfrac{\mathfrak{m}(\mathcal{S})}{n},
	\end{equation}
	and
	\begin{equation}\label{Ch3ConEquivJumPro2}
		\nu_{1} > 1 - \dfrac{\mathfrak{m}(\mathcal{S})}{(n - 1) + 1} = 1 - \dfrac{\mathfrak{m}(\mathcal{S})}{n}.
	\end{equation}
	Again, without loss of generality, we may assume that $\alpha_{0} \leq \alpha_{1}$, thus for every $\nu_{0}$ there always exist $\nu_{1}$ such that $\nu_{0} \leq \nu_{1}$. Hence the condition (\ref{Ch3ConEquivJumPro1}) implies (\ref{Ch3ConEquivJumPro2}). Because $\alpha = \min \left\lbrace \alpha_{0}, \alpha_{1}\right\rbrace = \alpha_{0}$ the condition (\ref{Ch3ConEquivJumPro1}) and (\ref{Ch3EqCndMar}) are the same. Therefore, problem (\ref{Ch3EqJumPro}) is solvable if condition (\ref{Ch3EqCndMar}) is satisfied, and the proof is complete.
\end{proof}
We should note that in this case we actually have four solvability conditions that can be written as 
\begin{eqnarray*}
	\nu_{j} > 1 - \dfrac{\mathfrak{m}^{\pm}(\mathcal{S})}{n}, & j = 0,1.
\end{eqnarray*} 
When all these conditions are fulfilled, then we have the four solutions given by (\ref{Ch3EqSolCuater2nd}) where
\begin{eqnarray*}
	\Upsilon^{l}_{0}(x) = \phi^{l}_{0} + \int\limits_{\mathbb{R}^{n}}\mathcal{D}_{n-1} \phi^{l}_{0}(y)E_{n-1}(y-x)dV(y),  & l = 0, 1;
\end{eqnarray*}
\begin{eqnarray*}
	\Upsilon^{l}_{1}(x) = \phi^{l}_{1} + \int\limits_{\mathbb{R}^{n}}E_{n-1}(y-x)\mathcal{D}_{n-1} \phi^{l}_{1}(y)dV(y),  &  l = 0, 1;
\end{eqnarray*}
here $\phi^{0}_{j}(z) = u_{j}(z)\chi^{+}(z)$ and  $\phi^{1}_{j}(z) = u_{j}(z)\chi^{-}(z)\rho(z)$, where for $j = 0, 1$; $u_{j}(z)$ is a Whitney extension of $\overline{\widehat{g}_{0}(\ux)}$ and $\widehat{g}_{1}(\ux)$, respectively, to the entire space $\mathbb{R}^{n}$, $\chi^{+}(z)$ is the characteristic function of $\Omega^{+}$, $\chi^{-}(z) = -\chi^{*}(z)$, $\chi^{*}(z)$ is the characteristic function of $\Omega^{*}$ and $\rho(z)$ is the real valued smooth function with compact support that was defined in the proof of \cite[Theorem \ref{TheoSolvCond}]{TAB21}. 

We combine each of the two functions $\Upsilon^{l}_{0}(x)$ in (\ref{Ch3EqSolCuater2nd}) with each of the two possible values of $\Upsilon^{l}_{1}(x)$. In contrast, if in the vectorial approach, we repeat the method developed in \cite{TAB21} to solve the scalar additive Jump Problem in the paravectorial setting, we only obtain at most two ways for computing the solutions.

The following example also illustrates this idea. In order to do that, we restrict ourselves to the vectorial Clifford Analysis case for $n = 2$, i.e. the study of functions from $\mathbb{R}^{2}$ to $\mathcal{C}\ell(2)$. Hence the Dirac operator have the form $\partial_{2} = e_{1}\displaystyle\frac{\partial}{\partial x_{1}} + e_{2}\displaystyle\frac{\partial}{\partial x_{2}}$.

\subsubsection*{Example}

A family of curves that will serve as the boundary of the problem, which is adapted from that presented in \cite{1DKats16} can be constructed. Nevertheless, the basic idea goes back as far as \cite{BKats83} and generalized to higher dimensions in \cite{TamayoCastro2022, TamayoCastro2023}. We include here the sketch of the construction for completeness.

We start by considering the two-dimensional square $Q = [0, 1]\times[-1, 0]$. The main idea is to add, to this square, infinitely many  rectangles with appropriate lengths and widths. To this end, let us fix $\alpha \geq 1$ and $\beta \geq 1$. The segment $[0,1]$ on the $x_{1}$ axis is conveniently divided. We decompose it into the subsegments $[2^{-m},2^{-m+1}]$ for each $m \in \mathbb{N}$, and subdivide each of these subsegments into $2^{[m\beta]}$ equally spaced segments, where $[m\beta]$ stands for the integer part of $m\beta$. We denote by $y_{mj}$ the endpoints to the right of these segments, where $j = 1, 2, ..., 2^{[m\beta]}$. Furthermore, the distance between two consecutive points $y_{mj}$ and $y_{m(j+1)}$ is $a_{m} = 2^{-m - [m\beta]}$, and let $B_{m} =\frac{1}{2}a_{m}^{\alpha}$. Then let the rectangles $R_{mj}$ be 
\begin{equation*}
	R_{mj} =  [y_{mj} - B_{m}, y_{mj}]\times[0, 2^{-m}].
\end{equation*}
Thus, we set 
\begin{equation*}
	T_{\alpha, \beta} := Q\bigcup\left(\bigcup_{m = 1}^{\infty}\bigcup_{j = 1}^{2^{[m\beta]}}R_{mj}\right).	
\end{equation*}
The boundaries of the corresponding $T_{\alpha, \beta}$ are the curves $\mathcal{O}_{\alpha, \beta}$, see Figure \ref{fig:curva}. 
\begin{figure}[ht]
	\centering
	\includegraphics[width=0.7\linewidth]{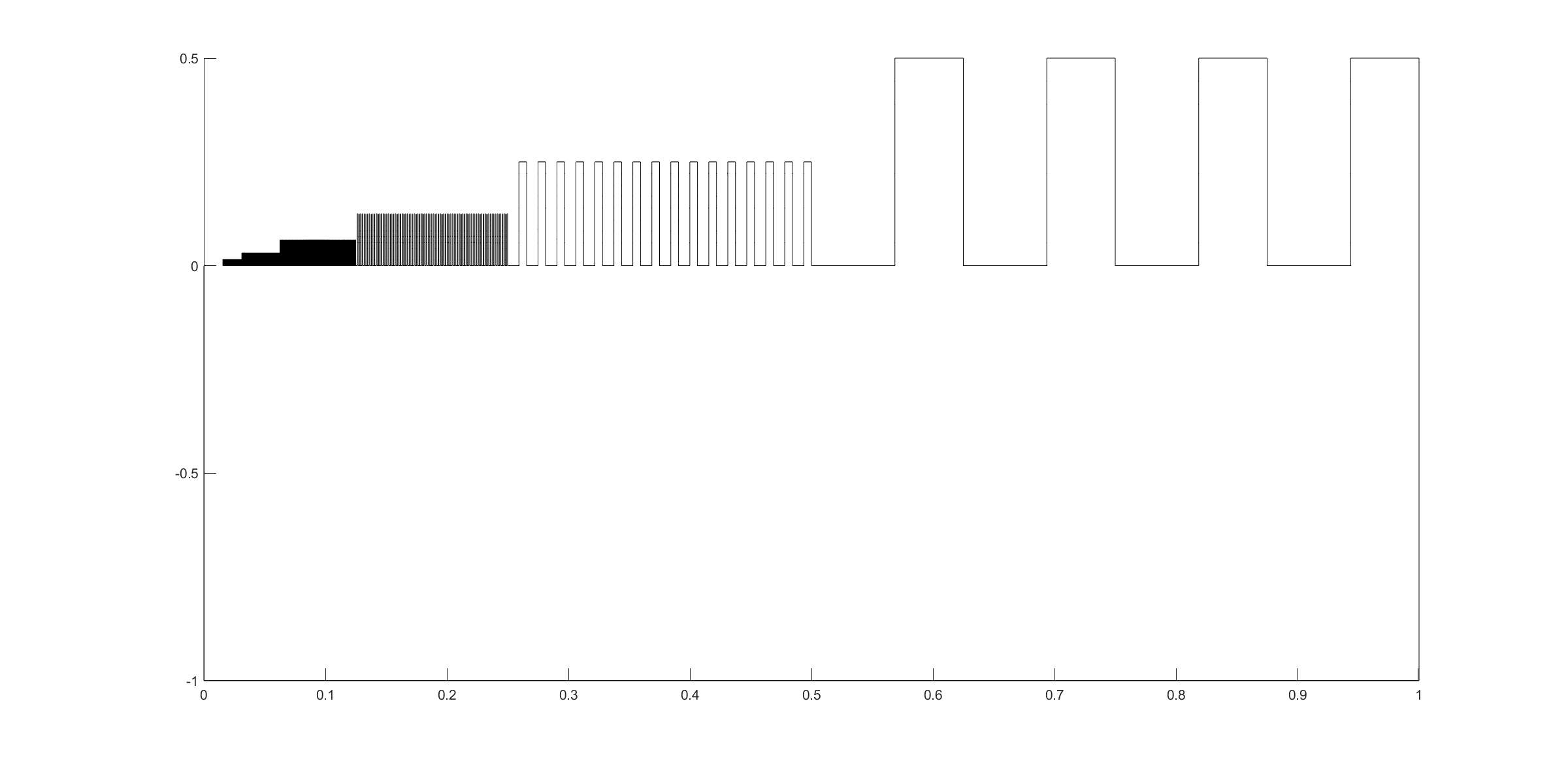}
	\caption[]{Curve $\mathcal{O}_{\alpha, \beta}$ for $\beta = 2.2 $ and $\alpha = 1.05$}
	\label{fig:curva}
\end{figure}
We follow \cite{1DKats16}, to see that for every $\alpha \geq 1$ and $\beta \geq 1$, the values of the inner and outer Marcinkiewicz exponent, are $\mathfrak{m}^{+}(\mathcal{O}_{\alpha, \beta}) = 1 - \displaystyle\frac{\beta - 1}{(\beta+1)\alpha}$ and  $\mathfrak{m}^{-}(\mathcal{O}_{\alpha, \beta}) = \displaystyle\frac{2}{\beta+1}$, respectively.

Now we consider the curves $\mathcal{O}_{\alpha, \beta}$ and the function 
$$g(x) = x_{1}x_{2} + (x_{1} + x_{2})e_{1}+ (x_{1} - x_{2})e_{2} + f(x_{1})e_{12},$$ 
where
\begin{equation*}
	f(\xi) = \left\lbrace \begin{array}{cc}
		\xi^{\frac{\beta}{\beta + 1}} \log\xi,	& \mathrm{if} \,\,\, \xi \neq 0, \\
		0,	& \mathrm{if} \,\,\, \xi = 0.
	\end{array}\right. 
\end{equation*}
The task is now to find a monogenic function $\Phi$ on $\R^{2}\setminus\mathcal{O}_{\alpha, \beta}$ continuously extendable from $\Omega^{\pm}$ to $\mathcal{O}_{\alpha, \beta}$ such that its boundary values $\Phi^{\pm}$ on $\mathcal{O}_{\alpha, \beta}$ satisfy the following conditions
\begin{equation}\label{ExJumpProb}
	\begin{array}{cc}
		\Phi^{+}(x) - \Phi^{-}(x) = g(x), & x \in \mathcal{O}_{\alpha, \beta}, \\
		\Phi^{-}(\infty) = 0. &   
	\end{array}
\end{equation}
Let us first show that for every $\nu < \displaystyle\frac{\beta}{\beta + 1}$ function $g \in C^{0,\nu}(\mathcal{O}_{\alpha, \beta})$, while $g \notin C^{0,\frac{\beta}{\beta + 1}}(\mathcal{O}_{\alpha, \beta})$.

For this purpose, we shall employ a similar result, given in \cite{Fiorenza2016}, for the function $f$ on the unitary segment, i.e. $f \in C^{0,\nu}([0,1])$ for every $\nu < \displaystyle\frac{\beta}{\beta + 1}$, while $f \notin C^{0,\frac{\beta}{\beta + 1}}([0,1])$. 

Then, the triangular inequality  in conjunction with standard properties of H\"older functions, we get
\begin{equation*}
	\left| g(x) - g(y) \right| < \left| f(x_{1}) - f(y_{1}) \right| + C\left| x - y \right|. 
\end{equation*}
Therefore,
\begin{eqnarray*}
	\left| g(x) - g(y) \right| < C\left| x - y \right|^{\nu},
\end{eqnarray*}
for all $\nu < \frac{\beta}{\beta + 1}$.

Now, we suppose that $g \in C^{0,\nu}(\mathcal{O}_{\alpha, \beta})$ where $1 \geq \nu \geq \displaystyle\frac{\beta}{\beta + 1}$, then
\begin{eqnarray*}
	\left| f(x) - f(y) \right| < \left| g(x) - g(y) \right| < C\left| x - y \right|^{\nu}, 
\end{eqnarray*}
which just means that
\begin{eqnarray*}
	\left|f(x_{1}) - f(y_{1}) \right| < C\left| x_{1} - y_{1} \right|^{\nu}, 
\end{eqnarray*}
which is a contradiction.

Analogously, $g_{0} \in C^{0,\nu}(\mathcal{O}_{\alpha, \beta})$ for every $\nu < \frac{\beta}{\beta + 1}$, while $g_{0} \notin C^{0,\frac{\beta}{\beta + 1}}(\mathcal{O}_{\alpha, \beta})$. However, $g_{1} \in C^{0,1}(\mathcal{O}_{\alpha, \beta})$.

Then, for any $\alpha \geq 1$ and $\beta \geq 1$, the function $g$ does not satisfy the condition 
\begin{eqnarray*}
	\nu > 1 - \dfrac{\mathfrak{m}^{-}(\mathcal{O}_{\alpha, \beta})}{2},
\end{eqnarray*} 
and in a similar way the even part $g_{0}$ do not satisfy the condition
\begin{eqnarray*}
	\nu_{0} > 1 - \dfrac{\mathfrak{m}^{-}(\mathcal{O}_{\alpha, \beta})}{2}.
\end{eqnarray*} 
However, the odd part $g_{1}$ satisfy both conditions
\begin{eqnarray*}
	\nu_{1} > 1 - \dfrac{\mathfrak{m}^{\pm}(\mathcal{O}_{\alpha, \beta})}{2}.
\end{eqnarray*} 
Therefore, if we repeat the procedure developed in \cite{TamayoCastro2022} to prove Theorem \ref{Ch3TheoSolvCond} we will get only one representation of the solution. In contrast, using this method, we get two representations of the solution given by (\ref{Ch3EqSolCuater2nd}) where
\begin{eqnarray*}
	\Upsilon_{0}(x) = \phi_{0} + \int\limits_{\mathbb{R}^{2}}\mathcal{D}_{1} \phi_{0}(y)E_{1}(y-x)dV(y),
\end{eqnarray*}
\begin{eqnarray*}
	\Upsilon^{l}_{1}(x) = \phi^{l}_{1} + \int\limits_{\mathbb{R}^{2}}E_{1}(y-x)\mathcal{D}_{1} \phi^{l}_{1}(y)dV(y),  &  l = 0, 1;
\end{eqnarray*}
here $\phi_{0}(z) = v(z)\chi^{+}(z)$; $v(z)$ is a Whitney extension of $\overline{\widehat{g}_{0}(\ux)}$ to the entire plane, $\chi^{+}(z)$ is the characteristic function of $\Omega^{+}$. While, here $\phi^{0}_{1}(z) = u(z)\chi^{+}(z)$ and  $\phi^{1}_{1}(z) = u(z)\chi^{-}(z)\rho(z)$; $u(z)$ is a Whitney extension of $\widehat{g}_{1}(\ux)$ to the plane, again $\chi^{+}(z)$ is the characteristic function of $\Omega^{+}$, and $\chi^{-}(z) = -\chi^{*}(z)$, $\chi^{*}(z)$ is the characteristic function of $\Omega^{*}$ and once more $\rho(z)$ is the real valued smooth function with compact support that was defined in the proof of \cite[Theorem \ref{TheoSolvCond}]{TAB21}.

If $\Upsilon^{0}_{1}(x) \neq \Upsilon^{1}_{1}(x)$, then we would have two different solutions that are not provided for the method developed in \cite{TamayoCastro2022} for paravectorial Clifford Analysis. In the following theorem, we will present a unicity condition for the solution of problem (\ref{Ch3EqJumPro}).

\begin{theorem}\label{Ch3TheoUniMar}
	Let $\mathcal{S}$ be a topologically compact surface which is the boundary of a Jordan domain in $\mathbb{R}^{n}$, and let $f \in C^{0,\nu}(\mathcal{S})$, with $\nu > 1 - \dfrac{\mathfrak{m}(\mathcal{S})}{n}$ and 
	\begin{equation}\label{Ch3EqUniCndMar}
		\dim_{H}\mathcal{S} - (n - 1) < \mu < 1 - \dfrac{n(1 - \nu)}{\mathfrak{m}(\mathcal{S})}.
	\end{equation}
	Then the solution to the scalar additive Jump Problem (\ref{Ch3EqJumPro}) is unique in the classes $C^{0,\mu}(\overline{\Omega^{+}})$ and $C^{0,\mu}(\overline{\Omega^{-}})$.
\end{theorem}
In the same manner as in Theorem \ref{TheoUniMar}, the unicity in Theorem \ref{Ch3TheoUniMar} is assumed when there exists a value of $\mu$ such that condition (\ref{Ch3EqUniCndMar}) is fulfilled.
\begin{proof}
	Working in the same way for unicity conditions, we have that the solution is unique in the classes $C^{0,\mu}(\overline{\Omega^{+}})$ and $C^{0,\mu}(\overline{\Omega^{-}})$ where $\mu$ must satisfy the next inequalities,
	\begin{equation}\label{Ch3EqUniCndMarPr0}
		\dim_{H}\mathcal{S} - (n - 1) < \mu < 1 - \dfrac{n(1 - \nu_{0})}{\mathfrak{m}(\mathcal{S})},
	\end{equation}
	and 
	\begin{equation}\label{Ch3EqUniCndMarPr1}
		\dim_{H}\mathcal{S} - (n - 1) < \mu < 1 - \dfrac{n(1 - \nu_{1})}{\mathfrak{m}(\mathcal{S})},
	\end{equation}
	We assume one time more and causing no loss of generality, that $\alpha_{0} \leq \alpha_{1}$. Thus for every $\nu_{0}$ always exist $\nu_{1}$ such that $\nu_{0} \leq \nu_{1}$. Hence, we have that
	\begin{equation*}
		1 - \dfrac{n(1 - \nu_{0})}{\mathfrak{m}(\mathcal{S})} \leq 1 - \dfrac{n(1 - \nu_{1})}{\mathfrak{m}(\mathcal{S})}.
	\end{equation*}
	Consequently, the condition (\ref{Ch3EqUniCndMarPr1}) holds when condition (\ref{Ch3EqUniCndMarPr0}) is fulfilled. Accordingly, the solution to the scalar additive Jump Problem (\ref{Ch3EqJumPro}) is unique in the classes $C^{0,\mu}(\overline{\Omega^{\pm}})$, where 
	\begin{equation*}
		\dim_{H}\mathcal{S} - (n - 1) < \mu < 1 - \dfrac{n(1 - \nu)}{\mathfrak{m}(\mathcal{S})},
	\end{equation*}
	and the proof is complete.
\end{proof}
We should note that if the representations of the solutions give two different functions this means that 
\begin{equation*}
	\dim_{H}\mathcal{S} - (n - 1)  \geq 1 - \dfrac{n(1 - \nu)}{\mathfrak{m}(\mathcal{S})}.
\end{equation*}

\subsection{Case of null odd part in lower dimensions}
This subsection deals with the particular (but important) case of \eqref{ProbOrig} when we consider functions $f: \mathbb{R}^{2} \rightarrow \mathcal{C}\ell(2)$.  When $G_{1} \equiv 0$ problem (\ref{CondComp2}) turns into two independent RBVPs in complex analysis with the same coefficient $\overline{\widehat{G}_{0}(\ux)}$. 
\begin{equation}\label{Ch3EqSysLowFrac}
	\Big\lbrace \begin{array}{c}
		\Upsilon_{0}^{+}(\ux) - \overline{\widehat{G}_{0}(\ux)}\Upsilon_{0}^{-}(\ux) = \overline{\widehat{g}_{0}(\ux)} \\
		\Upsilon_{1}^{+}(\ux) -\overline{\widehat{G}_{0}(\ux)}\Upsilon_{1}^{-}(\ux) = \widehat{g}_{1}(\ux).
	\end{array} 
\end{equation}
Due to the fact that $G(\ux) \neq 0$ then $\overline{\widehat{G}_{0}(\ux)} \neq 0$. Defining $Ind(G):= Ind(\widehat{G}_{0})= \aleph$, yields $Ind(\overline{\widehat{G}_{0}}) = -\aleph$. Following the standard techniques of the RBVP theory, we are able to reduce the solvability of this problem to that of the scalar additive Jump problem. In \cite{1DKats16}, this has previously been used for the RBVP in complex analysis.

\begin{equation*}
	\begin{array}{cc}
		X^{+}(z) = e^{\varGamma(z)},	& X^{-}(z) = z^{\aleph}e^{\varGamma(z)},
	\end{array} 
\end{equation*}
where $\varGamma(z)$ is the solution of the problem
\begin{eqnarray*}
	\varGamma^{+}(t) - \varGamma^{-}(t) = \log[t^{\aleph}\overline{\widehat{G}_{0}(t)}], & t \in \mathcal{S},
\end{eqnarray*} 
i.e.
\begin{equation}\label{EqPartSolRBVCiffFracGamma}
	\varGamma(t) = f - \dfrac{1}{2\pi i}\iint\limits_{\mathbb{C}}\dfrac{\partial f}{\partial\overline{\zeta}}\dfrac{d\zeta d\overline{\zeta}}{(\zeta - z)}
\end{equation}
in which $f$ is either $u(z)\chi^{+}(z)$ or $u(z)\chi^{-}(z)\rho(z)$, where $u$ denotes a Whitney extension of $\log[t^{\aleph}\overline{\widehat{G}_{0}(t)}]$ to the whole complex plane. Again, $\chi^{+}(z)$ is the characteristic function of $\Omega^{+}$, $\chi^{-}(z) = -\chi^{*}(z)$, $\chi^{*}$ is the characteristic function of $\Omega^{*}$ and $\rho(z)$ is the smooth function with compact support defined in the proof of Theorem \ref{TheoSolvCond}, see \cite{TAB21}.

The functions $X^{\pm}$, commonly called canonical functions, see \cite{Gajov,Lu,Mu}, fulfill the following relation,
\begin{equation*}
	\dfrac{X^{+}(t)}{X^{-}(t)} = \overline{\widehat{G}_{0}(t)}.
\end{equation*}
Therefore, problem (\ref{Ch3EqSysLowFrac}) can be rewrote as
\begin{equation*}
	\left\lbrace  \begin{array}{c}
		\dfrac{\Upsilon_{0}^{+}(t)}{X^{+}(t)} - \dfrac{\Upsilon_{0}^{-}(t)}{X^{-}(t)} = \dfrac{\overline{\widehat{g}_{0}(t)}}{X^{+}(t)} \\
		\dfrac{\Upsilon_{1}^{+}(t)}{X^{+}(t)} - \dfrac{\Upsilon_{1}^{-}(t)}{X^{-}(t)} = \dfrac{\widehat{g}_{1}(t)}{X^{+}(t)} \\
	\end{array}. \right. 
\end{equation*}
Let functions $\Phi^{0}_{i}$ given by
\begin{eqnarray}\label{EqPartSolvRVBCliffFract}
	\Phi^{0}_{i} = \phi_{i} - \dfrac{1}{2\pi i}\iint\limits_{\mathbb{C}}\dfrac{\partial\phi_{i}}{\partial\overline{\zeta}}\dfrac{d\zeta d\overline{\zeta}}{\zeta - z} & i = 0,1
\end{eqnarray}
here $\phi_{i}(z)$ is either $u_{i}(z)\chi^{+}(z)$ or $u_{i}(z)\chi^{-}(z)\rho(z)$, where $u_{i}(z)$ are a Whitney extension of $\dfrac{\overline{\widehat{g}_{0}(\ux)}}{X^{+}(t)}$ and $\dfrac{\widehat{g}_{1}(\ux)}{X^{+}(t)}$, to the entire complex plane, respectively and the other functions involved are the same as in (\ref{EqPartSolRBVCiffFracGamma}).

Therefore, we have 
\begin{theorem}
	Suppose that, in problem (\ref{ProbOrig}), the functions $G(\ux)$ and $g(\ux)$ are H\"older continuous with exponent $\nu$ satisfying the condition $\nu > 1 - \frac{1}{2}\mathfrak{m}(\mathcal{S})$ and $G(\ux)$ is nonzero. If a solution is sought in the class of H\"older continuous functions with exponent $\mu$ in $\overline{\Omega^{+}}$ and $\overline{\Omega^{-}}$ and $\mu$ satisfies (\ref{EqUniCndMar}) then for $\aleph \leq 1$ the general solution is obtained by (\ref{Ch3EqSolCuater2nd}) where 
	\begin{equation*}
		\Upsilon_{0}(z) = X(z)[\Phi^{0}_{0}(z) + P_{0}(z)],
	\end{equation*}
	\begin{equation*}
		\Upsilon_{1}(z) = X(z)[\Phi^{0}_{1}(z) + P_{1}(z)],
	\end{equation*}
	where $P_{0}$ and $P_{1}$ are two polynomials of degree at most $-\aleph$;	for $\aleph = 1$ we put $P_{0} \equiv 0, P_{1} \equiv 0$;	for $\aleph > 1$ the problem has $2(\aleph - 1)$ solvability conditions.
\end{theorem}
Here we do not write down explicitly the solvability conditions, which can be obtained by expanding the integrals of (\ref{EqPartSolvRVBCliffFract}) in a power series at $\infty$.


\subsection*{Conflict of interest}
The authors have no known competing financial interests or personal relationships that could have appeared to influence the work reported in this paper.

\subsection*{ORCID}
\noindent
Carlos Daniel Tamayo-Castro: https://orcid.org/0000-0002-5584-8274\\
Juan Bory-Reyes: https://orcid.org/0000-0002-7004-1794\\
Ricardo Abreu-Blaya: https://orcid.org/0000-0003-1453-7223

\bibliographystyle{plain}
\bibliography{BiblioBoundDecoMar_Preprint}
\end{document}